\documentclass{article}
\usepackage{graphicx} 

\usepackage[utf8]{inputenc}
\usepackage[T1]{fontenc}
\usepackage[utf8]{inputenc}
\usepackage{newtxtext}
\usepackage{latexsym}
\usepackage{mathrsfs}
\usepackage{typearea}
\usepackage{indentfirst} 
\typearea{12}
\newenvironment{reptheorem}[1]
{\begin{trivlist}
\item[\hskip\labelsep{\bfseries Theorem~\ref{#1}.}]\itshape}
{\end{trivlist}}

\usepackage{etoolbox}

\apptocmd{\thebibliography}{%
  \setlength{\itemsep}{0pt}
  \setlength{\parsep}{0pt}
  \setlength{\parskip}{0pt}
}{}{}

\usepackage{amsmath}
\usepackage{cases}
\usepackage{amssymb}
\usepackage{mathrsfs}
\usepackage{color}
\usepackage{usebib}
\usepackage{comment}
\usepackage{hyperref}
\usepackage{cleveref}
\hypersetup{%
 setpagesize=false,%
 bookmarks=true,%
 bookmarksdepth=tocdepth,%
 bookmarksnumbered=true,%
 colorlinks=false,%
 pdftitle={},%
 pdfsubject={},%
 pdfauthor={},%
 pdfkeywords={}}
\usepackage[dvipsnames]{xcolor}
\hypersetup{hypertexnames=false,colorlinks=true,
            linkcolor=NavyBlue,
            citecolor=NavyBlue,
            urlcolor=NavyBlue
}

\newenvironment{repcorollary}[1]
{\begin{trivlist}
\item[\hskip\labelsep{\bfseries Corollary~\ref{#1}.}]\itshape}
{\end{trivlist}}

\makeatletter

\@addtoreset{equation}{section}
\makeatother

\usepackage{amsthm}
\newtheorem{thm}{Theorem}[section]

\newtheorem{lem}[thm]{Lemma}
\newtheorem{prp}[thm]{Proposition}
\newtheorem{rmk}[thm]{Remark}

\newtheorem{cor}[thm]{Corollary}

\newtheorem{conj}[thm]{Conjecture}

\newtheorem{thm'}{Theorem}
\newtheorem{conj'}[thm']{Conjecture}
\newtheorem{ques'}[thm']{Question}

\usepackage{cleveref}

\renewenvironment{abstract}{%
  \small
  \begin{quotation}
}{%
  \end{quotation}
}

\newcommand{\Z}{\mathbb{Z}}
\newcommand{\Wh}{\mathrm{Wh}}

\title{Unknotting number, ribbon concordance, and singular instantons}
\author{
Hayato Imori
\and
JungHwan Park
\and
Masaki Taniguchi
}
\date{}

\begin{document}

\maketitle

\begin{abstract}
We use equivariant singular instanton Floer theory with the Chern--Simons filtration to obstruct same-sign unknotting operations. We show that, for a large class of slice knots obtained through ribbon concordance, any unknotting sequence of null-homologous twists must contain both signs. The same method gives a $3$--manifold analogue, obstructing certain homology $3$--spheres from surgery on a knot and providing evidence for the monotonicity of the Dehn surgery number under ribbon homology cobordism.
\end{abstract}

\section{Introduction}

The \emph{unknotting number} $u(K)$ of a knot $K$ is the minimal number of crossing changes required to transform $K$ into the unknot. Although its definition is simple and classical~\cite{Wendt}, the invariant remains poorly understood and surprisingly subtle.  For example, Brittenham and Hermiller~\cite{BrittenhamHermiller} recently showed that the unknotting number is not additive under connected sum, disproving a long-standing conjecture.

A celebrated theorem of Scharlemann~\cite{Scharlemann1985} states that any knot with unknotting number one is prime; equivalently, every composite knot has unknotting number greater than one.  The following conjecture predicts a strengthening of Scharlemann's theorem and serves as one of the main motivations for this article.

\begin{conj}\label{conj:band}
Any band sum of two nontrivial knots has unknotting number greater than one.
\end{conj}

This conjecture is also closely related to a problem in Kirby's problem
list~\cite[Problem~1.53]{K3}, which asks which knot invariants are monotone
under ribbon concordance.  Given two knots $K_-$ and $K_+$, recall that $K_-$ is
\emph{ribbon concordant} to $K_+$ if there is a smooth concordance from $K_-$ to $K_+$ with no local maxima.  This relation was introduced by Gordon~\cite{GordonRibbon}; Agol recently proved that it defines a partial order on knots in $S^3$~\cite{AgolRibbon}. To the best of the authors' knowledge, it is not known whether the unknotting number is monotone under ribbon concordance; that is, whether the existence of a ribbon concordance from $K_-$ to $K_+$ implies $$u(K_-)\leq u(K_+).$$ Since there is a ribbon concordance from $K_1\mathbin{\#} K_2$ to any band sum of $K_1$ and $K_2$ by Miyazaki~\cite{Miyazaki:1998}, monotonicity of the unknotting number under ribbon concordance would imply Conjecture~\ref{conj:band} from Scharlemann's theorem.

In this article, we build on recent advances in equivariant singular instanton Floer theory with a Chern--Simons filtration~\cite{DS19, DS20, Ha21, DISST} and provide supporting evidence for Conjecture~\ref{conj:band} as follows. Recall that for each parameter
\begin{align} \label{admissible}
\omega\in
\left(0,\frac12\right)
\cap
\left\{
\frac{m}{2p^n}
:
m\in\mathbb Z,\;
p\text{ is prime},\;
n\in\mathbb Z_{>0}
\right\},
\end{align}
there is an instanton local class homomorphism
\[
\Omega^{\mathfrak{E}, \omega}\colon \mathcal{C} \to
\Theta^{\mathfrak{E}, \omega}_{R^{\omega}},
\]
from the smooth knot concordance group to the group of local equivalence
classes of enriched $S$-complexes, constructed in
\cite[Sections~5 and~7]{DISST}.

 Knots with nontrivial $\Omega^{\mathfrak{E},\omega}$--invariant include those representing elements of infinite order in Levine's algebraic knot concordance group~\cite{Levine1969KnotCobordism}, as well as squeezed knots with nonzero Rasmussen invariant~\cite{FellerLewarkLobbSqueezed}; see Proposition~\ref{prp:basic-nontrivial-Omega} for details.  In particular, this includes all non-slice quasipositive knots; see \cite{Rudolph1993} for the definition of quasipositivity. Finally, recall that each crossing change is assigned a sign, either positive or negative.

\begin{thm}\label{thm:main1}
Let $K$ be a slice knot, and suppose that there is a ribbon concordance from $K_1\# K_2$ to $K$. If one of the summands, say $K_1$, satisfies
$\Omega^{\mathfrak{E}, \omega}([K_1]) \neq [0]$
for some $\omega$ satisfying \eqref{admissible}, then any sequence of crossing changes transforming $K$ into the unknot must contain both positive and negative crossing changes.

In particular, if a slice knot $K$ is obtained as a band sum of two nontrivial knots and one of the summands represents a nontrivial class in $\Omega^{\mathfrak{E}, \omega}$, then $K$ has unknotting number greater than one.
\end{thm}

A few remarks are in order. We begin with an example showing that the conclusion of Theorem~\ref{thm:main1} does not hold in complete generality. Let $K=4_1\mathbin{\#}4_1$, where $4_1$ denotes the figure-eight knot. Since $4_1$ is amphichiral, the knot $K$ is slice. Moreover, since $4_1$ is amphichiral and has unknotting number one, it can be unknotted by either a positive or a negative crossing change. Consequently, $K$ can be unknotted by two positive crossing changes, as well as by two negative crossing changes. Thus, $K$ does not satisfy the conclusion of Theorem~\ref{thm:main1}. More generally, analogous examples can be obtained by replacing $4_1$ with an amphichiral knot of unknotting number one.

Second, the unknotting number admits the following natural generalization. The
\emph{untwisting number} of a knot $K$, denoted by $tu(K)$, is the minimal
number of null-homologous twists required to transform $K$ into the unknot. A
\emph{null-homologous twist} is obtained by choosing an unknotted curve
$C \subset S^3 \smallsetminus K$ satisfying $\operatorname{lk}(C,K)=0$ and performing
$\pm 1$-surgery on $C$; see~\cite{MathieuDomergue1988, 
InceUntwisting,InceUntwistingHF,LivingstonNullHomologousUnknottings,
McCoyNullHomologousTwisting,AllenInceKimRuppikTurner2024}
for related results. We call the twist
\emph{positive} if the surgery coefficient is $+1$, and \emph{negative} if the
surgery coefficient is $-1$. We use the analogous convention for crossing
changes: a positive crossing change changes a negative crossing into a positive crossing. The conclusion of Theorem~\ref{thm:main1} extends to this setting. Namely, under the hypotheses of Theorem~\ref{thm:main1}, every sequence of null-homologous twists transforming $K$ into the unknot must contain both positive and negative twists.

Finally, the knots under consideration are slice, and hence any obstruction to
the unknotting number which factors through the smooth $4$-ball genus vanishes;
for instance, the bounds coming from $\tau$ and $s$ give no information
\cite{OzsvathSzabo2003,Rasmussen2010}. Moreover,
Theorem~\ref{thm:main1} applies to many knots with trivial Alexander
polynomial. For such knots, the standard classical algebraic lower bounds for the
unknotting number also vanish; see, for example,
\cite{Murakami1990,Nakanishi1981,Lickorish1985,
BorodzikFriedlClassicalI,BorodzikFriedl2014,FellerLewarkClassicalUpper,
FKLNP:2022,FellerLewarkBalanced}.\footnote{In fact, if the summand $K_1$ has infinite order in the algebraic knot concordance group, then the crossing-change conclusion of
Theorem~\ref{thm:main1} also follows from
\cite[Theorem~4.1]{BorodzikFriedlClassicalI} and
\cite[Proposition~3.4]{FKLNP:2022}.
} Here and throughout, for a knot $K$, we denote by $-K$ the knot obtained by taking the mirror image of $K$ and reversing its orientation. This knot represents the inverse of $K$ in the smooth knot concordance group. For concrete examples, let $\Wh$ denote the positive Whitehead double pattern.
If $K$ is strongly quasipositive, then $\Wh(K)$ is also strongly quasipositive
\cite{Rudolph1993}. Moreover, $\Wh(K)$ can be unknotted by changing a positive
clasp crossing to a negative crossing. This gives the following consequence.

\begin{cor}\label{cor:whitehead-double}
Let $K$ be a non-slice strongly quasipositive knot, and set
$J_K=\Wh(K)\#-\Wh(K)$. Then any sequence of null-homologous twists
transforming $J_K$ into the unknot must contain both positive and negative
twists. In particular, $tu(J_K)=2$.
\end{cor}

Our work is reminiscent of Auckly's work~\cite{Auckly,AucklyHyperbolic};
see also~\cite{SatoTaniguchi2020}.  Auckly constructed two irreducible homology $3$--spheres, one toroidal and
one hyperbolic, which are not obtainable by surgery on a knot, using Taubes's
periodic-end theorem~\cite{T87}. Theorem~\ref{thm:main1} may be viewed as a knot-theoretic analogue of this circle of ideas.  In both settings, gauge-theoretic information carried by an initial object is propagated through a ribbon-type cobordism to certify that the resulting object cannot have a simple topological description. We discuss this in more detail below.

There has been substantial work on understanding which $3$--manifolds can, and
which cannot, be obtained by surgery on knots in $S^3$; see, e.g.,
\cite{Berge,Gordon-Luecke:1989,Boyer-Lines:1990,Auckly, AucklyHyperbolic,
OzsvathSzaboLens,HeddenBerge,GreeneLensSpace,HKL,Hom-Lidman:2018, HeddenKimMarkPark2019,
SatoTaniguchi2020}.  See also~\cite{GordonICM} and~\cite[Section~1.1]{HKL}
for excellent surveys.  Here we are mainly interested in homology $3$--spheres which cannot be realized as surgery on a knot in $S^3$.  Auckly's argument shows that certain homology $3$--spheres cannot bound smooth simply connected definite $4$--manifolds, and combines this with the fact that any homology $3$--sphere obtained by surgery on a knot in $S^3$ must bound such a $4$--manifold.  His examples start from a homology $3$--sphere $Y$ bounding a simply connected
definite $4$--manifold with non-diagonalizable intersection form, and pass to the
desired irreducible homology $3$--sphere $M$ through a cobordism from
$Y \mathbin{\#} -Y$ using one $1$--handle and one $2$--handle. Taubes's
periodic-end theorem is then used to obstruct a simply connected definite
filling of $M$.  The results below substantially enlarge the scope of this construction: in
particular, they can be applied even when the initial homology $3$--sphere $Y$ bounds
no definite $4$--manifold at all, a case in which Taubes's theorem cannot be
applied.

Recall that, for a closed connected oriented $3$--manifold $M$, the
\emph{Dehn surgery number} $DS(M)$ is the minimal number of components of a
link $L\subset S^3$ such that Dehn surgery on $L$ yields $M$.  A
$3$--manifold analogue of Scharlemann's theorem~\cite{Scharlemann1985} follows
from a theorem of Gordon and Luecke~\cite{Gordon-Luecke:1989}: reducible homology $3$--spheres cannot be obtained by surgery on knots.  Equivalently,
every reducible homology $3$--sphere has Dehn surgery number greater than
one.

An analogue of Conjecture~\ref{conj:band}, in the spirit of
\cite[Problem~1.53]{K3}, can be formulated for $3$--manifolds as follows.
Recall that, given homology $3$--spheres $Y_-$ and $Y_+$, a \emph{ribbon homology cobordism} from $Y_-$ to $Y_+$ is a homology cobordism admitting a handle decomposition relative to $Y_-$ with only $1$-- and $2$--handles.  Recent work
shows that this relation defines a partial order on the set of diffeomorphism
classes of closed oriented $3$--manifolds~\cite{HuberRibbon,
FriedlMisevZentner,Ghanem}; see also~\cite{DLVVW} for the original formulation.

\begin{conj}\label{ques:DS-monotone}
Let $Y_-$ and $Y_+$ be homology $3$--spheres.  Suppose there is a ribbon homology
cobordism from $Y_-$ to $Y_+$.  Then
\[
DS(Y_-)\leq DS(Y_+).
\]
In particular, if $Y_1$ and $Y_2$ are nontrivial homology $3$--spheres and there is
a ribbon homology cobordism from $Y_1 \mathbin{\#} Y_2$ to a homology $3$--sphere
$M$, then $M$ cannot be obtained by surgery on a knot in $S^3$.
\end{conj}

We give an instanton-theoretic obstruction which provides evidence for Conjecture~\ref{ques:DS-monotone} and, in particular, rules out one-component surgery descriptions for a certain class of homology $3$--spheres.  Let $\Theta^3_{\mathbb Z}$ denote the homology cobordism group of
oriented homology $3$--spheres. For a homology $3$--sphere $Y$, let $U\subset Y$ be an unknot contained in a $3$-ball. Applying the homology concordance homomorphism constructed in \cite[Section~5]{DISST} to the pair $(Y,U)$ defines a homomorphism
\[
\Omega\colon \Theta^3_{\mathbb Z}\longrightarrow
\Theta^{\mathfrak E}.
\]  Equivalently, $\Omega([Y])$ is
the local equivalence class of the enriched singular instanton $S$--complex
associated to $(Y,U)$, with holonomy parameter $\omega=1/4$. The following theorem is the $3$--manifold analogue of the knot-theoretic obstruction in Theorem~\ref{thm:main1}.

\begin{thm}\label{thm:DS-ribbon-obstruction}
Let $Y$ be a homology $3$--sphere representing the trivial element in
$\Theta^3_\Z$, and  suppose that there is a ribbon homology cobordism from
$Y_1 \mathbin{\#} Y_2$ to $Y$.  If one of the summands, say $Y_1$, satisfies
$\Omega([Y_1])\neq [0]$,
then $Y$ cannot be obtained by surgery on a knot in
$S^3$.
\end{thm}

Existing computations in instanton theory show that Theorem~\ref{thm:DS-ribbon-obstruction} applies to the following examples. In the following statement, a \emph{nontrivial integral linear combination} of
the homology $3$--spheres $Y_1,\ldots,Y_m$ means an expression
\[
c_1[Y_1]+\cdots+c_m[Y_m]\in\Theta^3_{\mathbb Z},
\qquad c_1,\ldots,c_m\in\mathbb Z,
\]
in which at least one coefficient $c_i$ is nonzero.

\begin{thm}\label{thm:example}
Let $Y$ be a homology $3$--sphere. Suppose that $Y$ is one of the following:
\begin{itemize}
\item A Seifert homology $3$--sphere with positive Fintushel--Stern
$R$--invariant \cite{FintushelStern1985}, or more generally, a nontrivial integral linear combination of
Seifert homology $3$--spheres detected by~\cite[Corollary~2.1]{Furuta1990HomologyCobordism};
\item A nontrivial integral linear combination of homology $3$--spheres of the form
$\{S^3_{1/n}(K)\}_{n\in  \Z_{>0}}$, where $K$ is a squeezed knot with positive Rasmussen invariant.
\end{itemize}
Then $\Omega([Y])\neq [0]$.
\end{thm}
We remark that Nozaki, Sato, and the third author~\cite[Theorem~1.5]{NST24}
constructed homology $3$--spheres $Y$ with $\Omega([Y])\neq [0]$ which do not bound
any definite $4$--manifold. Thus Auckly's method, which starts from a homology
sphere admitting a simply connected definite filling with non-diagonalizable
intersection form, would not apply to these examples.

We close the introduction by noting that our knot and homology $3$--sphere examples can be chosen prime and irreducible, respectively. In fact, the targets in both settings can be chosen hyperbolic.

\begin{rmk}
Silver--Whitten \cite[Theorem~2.2]{Silver-Whitten:2005} showed that
every knot admits ribbon concordances to hyperbolic knots of arbitrarily
large volume. Thus, if $\Omega^{\mathfrak E,\omega}([K])\neq[0]$, applying
Theorem~\ref{thm:main1} to the resulting ribbon concordances from
$K\mathbin{\#}-K$ gives infinitely many hyperbolic slice knots, each
of which cannot be unknotted by null-homologous twists all of the same
sign.

Similarly, let $Y$ be a homology $3$--sphere with $\Omega([Y])\neq[0]$.
Myers' theorem \cite[Theorem~1.1]{Myers:1993}, together with
Thurston's hyperbolic Dehn surgery theorem, gives, for infinitely many
framings, ribbon homology cobordisms from $Y\mathbin{\#}-Y$ to
hyperbolic homology $3$--spheres $M_i$; see also
\cite[proof of Corollary~4]{Aceto-Daemi-Hom-Lidman-Park:2022}. Theorem~1.5 therefore implies that $M_i$ cannot be obtained by surgery on a knot in $S^3$ for any $i$.
\end{rmk}

\subsection*{Acknowledgements} 
We would like to thank Peter Feller, Lukas Lewark, and Charles Livingston for
helpful comments. The first author was partially supported by the Samsung
Science and Technology Foundation (SSTF-BA2102-02), the NRF grant
RS-2025-00542968, and the Jang Young Sil Fellowship from KAIST. The second
author was partially supported by the Samsung Science and Technology Foundation
(SSTF-BA2102-02) and the NRF grant RS-2025-00542968. The third author was
partially supported by JSPS KAKENHI Grant Number 22K13921.

\section{Null-homologous disks with infinite cyclic complements}

We begin with the following proposition, which provides the key topological input for our arguments. For closely related constructions, see \cite[Section~4]{ConwayDaiMiller}; compare also \cite[Theorem~1.5]{ConwayMiller}.

\begin{prp}\label{prp:positive-twists-disk}
Suppose that $K$ can be transformed into the unknot by a sequence of $n$
positive null-homologous twists. Then $K$ bounds a smooth properly embedded
null-homologous disk
\[
D\subset \bigl(\#_n\overline{\mathbb{CP}}^{,2}\bigr)^\circ
\]
whose complement has infinite cyclic fundamental group. Moreover, the same
conclusion holds for any knot $J$ admitting a ribbon concordance to $K$.
\end{prp}

\begin{proof}
Reverse the given sequence of twists. Starting with the unknot, we obtain $K$
by successively attaching $n$ $(-1)$--framed $2$--handles along unknotted
curves in the complement of the current knot. Each attaching curve has linking
number zero with the current knot. Capping the initial unknot with the standard
disk in $B^4$ and following it through these cobordisms gives a smooth properly
embedded disk
\[
D\subset \bigl(\#_n\overline{\mathbb{CP}}^{\,2}\bigr)^\circ
\]
with boundary $K$. The algebraic intersection of $D$
with the homology class associated to each $2$--handle is the linking number of its attaching circle with the current knot. These linking numbers all
vanish, and thus $D$ is null-homologous.

We next compute the fundamental group of the complement. The complement of the standard slice disk for the unknot has fundamental group $\mathbb{Z}$, generated by a meridian. At each stage, attaching a $2$--handle adds the relation represented by its attaching circle. Since this circle has linking number zero with the current knot, this relation is trivial. Therefore, the resulting disk complement has infinite cyclic fundamental group.

Now suppose that there is a ribbon concordance from $J$ to $K$. Concatenating it with $D$ gives a properly embedded null-homologous disk with boundary $J$ in the same $4$--manifold. Moreover, by assumption, the complement of this disk is obtained from the original disk complement by attaching only $2$-- and $3$--handles. By Mayer--Vietoris, its first homology is $\mathbb{Z}$. Since the fundamental group before these handle attachments is infinite cyclic and the first homology remains isomorphic to $\mathbb Z$, each attaching circle represents the trivial element of the fundamental group. Hence the resulting disk complement also has infinite cyclic fundamental group.
\end{proof}

\begin{rmk}\label{rmk:ribbon-fillings}
There is an analogous observation for homology $3$--spheres. Suppose that
there is a ribbon homology cobordism from $Y$ to $M$ and that $M$ bounds a
simply connected smooth $4$--manifold. Gluing the reverse of the cobordism to
this filling produces a simply connected filling of $Y$. Indeed, viewed from
$M$, the cobordism has only $2$-- and $3$--handles, which cannot create
generators of the fundamental group. The resulting filling also has the same
intersection form as the original one; in particular, definiteness is
preserved. This is the same handle argument used in the proof of
Proposition~\ref{prp:positive-twists-disk}.
\end{rmk}

\section{Constraint from singular instanton theory}

Let $Y$ be an oriented homology $3$--sphere, and let $K \subset Y$ be an oriented knot with a positively oriented meridian $\mu_K$. Equivariant singular instanton Floer theory, developed in \cite{DS19,Ha21,DISST}, is, roughly speaking, an equivariant Morse--Floer theory for the Chern--Simons functional with prescribed meridional holonomy. To specify this holonomy, we fix a parameter $\omega \in (0,\frac{1}{2})$ such that
$$
\Delta_{Y,K}(e^{4\pi i\omega}) \neq 0,
$$
where $\Delta_{Y,K}(t)$ denotes the Alexander polynomial of the knot $K \subset Y$. 
In particular, this non-root condition is automatically satisfied whenever
\[
\omega=\frac{m}{2p^n}\in\left(0,\frac12\right),
\]
where $m\in\mathbb Z$, $n\in\mathbb Z_{> 0}$, and $p$ is prime; see
\cite[Section~7.2]{DISST}.\footnote{To avoid roots of Alexander polynomials for all knots, we require the numerator in \cite[Section~7.2]{DISST} to be even.}
The relevant character variety is
$$
\chi_{\omega}(Y,K)
:=
\left\{
\rho\colon \pi_1(Y\smallsetminus K)\longrightarrow SU(2)
\ \middle|\
\rho(\mu_K)\sim
\begin{pmatrix}
e^{2\pi i\omega} & 0\\
0 & e^{-2\pi i\omega}
\end{pmatrix}
\right\}\bigg/SU(2),
$$
where $\sim$ denotes conjugacy in $SU(2)$ and the quotient is taken with respect to the overall conjugation action. Thus, after suitable perturbations, the generators of the corresponding Floer complex are modeled by the critical points represented by elements of $\chi_{\omega}(Y,K)$. The construction associates to $(Y,K)$ a graded $\mathcal{S}$--complex
$$
\bigl(\widetilde{C}^{\omega}_{*}(Y,K),\widetilde{d}\bigr).
$$
Its coefficients may be taken over a ring $R_{\omega}$ depending on the holonomy parameter $\omega$. More precisely,
$$
R_{\omega}
=
\begin{cases}
\mathbb{Z}[T^{-1},T]\!][U^{\pm1}], & \omega\neq\frac14,\\[3pt]
\mathbb{Z}[T^{\pm1},U^{\pm1}], & \omega=\frac14.
\end{cases}
$$

The $\mathcal{S}$--complexes admit cobordism maps induced by suitable cobordisms of pairs. Let $W$ be a smooth compact connected cobordism from a homology $3$--sphere $Y$ to a homology $3$--sphere $Y'$, and let $S\subset W$
be a smoothly and properly embedded oriented surface from $K\subset Y$
to $K'\subset Y'$. It is called \emph{negative definite of strong height $i$} with respect to the holonomy parameter $\omega$ if
\begin{itemize}
\item $b_1(W)=b^+(W)=0$;
\item the minimal index of a reducible ASD connection is $2i-1$;
\item the associated element $\eta_{\omega}(W,S)$ is invertible in $R_{\omega}$.
\end{itemize}
Here, the index of the reducible connection $A_L$ associated with a splitting $E=L\oplus L^*$ into complex line bundles is given by
$$
\begin{aligned}
\operatorname{ind}(A_L)
={}&
8\kappa(A_L)
+2(1-4\omega)\nu(A_L)
-\frac{3}{2}\bigl(\chi(W)+\sigma(W)\bigr)
+\chi(S)
\\
&\quad
+8\omega^2 S\cdot S
+\sigma_\omega(Y,K)
-\sigma_\omega(Y',K')
-1,
\end{aligned}
$$
where $\kappa(A_L)$ is the topological energy
$$
\kappa(A_L)=-\left(c_1(L)+\omega S\right)^2,
$$
and $\nu(A_L)$ is the monopole number
$$
\nu(A_L)=-2c_1(L)\cdot S.
$$
We denote by $\kappa_{\min}^{\omega}(W,S)$ the minimal value of $\kappa(A_L)$ among reducibles $A_L$ of index $2i-1$.

A lift of the Chern--Simons functional to $\mathbb{R}$ induces a real filtration on the $\mathcal{S}$--complex. This filtered complex is referred to as the enriched $\mathcal{S}$--complex, which we denote by $\mathfrak{E}^{\omega}(Y,K)$. 
We use the notion of a \textit{local morphism} between enriched $\mathcal S$-complexes, as defined in \cite[Definition~5.5 (i) and (ii)]{DISST}. In our applications, these morphisms are induced by negative-definite cobordisms of strong height $0$ satisfying $\kappa_{\min}^{\omega}=0$. Two enriched $\mathcal{S}$--complexes are said to be \textit{locally equivalent} if there exist local morphisms in both directions between them. 
Relaxing the homotopy coherence conditions relating the sequences of local maps leads to the notion of a \textit{weak local morphism}, which corresponds to dropping condition~(ii) in \cite[Definition~5.5]{DISST}. Two enriched $\mathcal S$-complexes are \textit{weakly locally equivalent} if there exist weak local morphisms in both directions between them.
Throughout this paper, we work with the weak local equivalence class of the enriched $\mathcal S$-complex $\mathfrak E^\omega(Y,K)$, rather than its local equivalence class. Accordingly, $\Omega^{\mathfrak E,\omega}$ denotes the resulting invariant with values in weak local equivalence classes.

Following \cite{DISST}, the enriched $\mathcal{S}$--complex $\mathfrak{E}^{\omega}(Y,K)$ gives rise to a numerical invariant
$$
\mathcal{N}^{\omega}_{(Y,K)}
\colon
\mathbb{Z}\times[-\infty,0)
\longrightarrow
[0,\infty].
$$
We shall mainly use the $r_0$--invariant extracted from $\mathcal{N}^{\omega}_{(Y,K)}$. Following \cite[(99)]{DISST}, define its transpose by
$$
\mathcal{N}^{\omega,\top}_{(Y,K)}(k,r)
:=
\min\left\{
\inf\left\{
s\in[-\infty,0)
\ \middle|\
\mathcal{N}^{\omega}_{(Y,K)}(k,s)\leq r
\right\},
\,0
\right\}
$$
for $r<\infty$, and set
$$
\mathcal{N}^{\omega,\top}_{(Y,K)}(k,\infty)
:=
\lim_{r\to\infty}
\mathcal{N}^{\omega,\top}_{(Y,K)}(k,r).
$$
For $s\in[-\infty,0]$, define
\begin{align}\label{equation_r_s}
r_s^{\omega}(Y,K)
:=
-\mathcal{N}^{\omega,\top}_{(Y,K)}(0,-s).
\end{align}
In particular,
$$
r_0^{\omega}(Y,K)
=
-\mathcal{N}^{\omega,\top}_{(Y,K)}(0,0)
\in[0,\infty].
$$
When $Y=S^3$, we abbreviate
$$
r_s^{\omega}(K)
:=
r_s^{\omega}(S^3,K).
$$
By construction, these invariants depend only on the weak local equivalence class of the enriched $\mathcal{S}$--complex.

The following lemma will be used repeatedly. For an oriented knot $K\subset Y$, let $-Y$ denote $Y$ with its orientation reversed and let $-K\subset -Y$ denote $K$ with its orientation reversed. The pair $(-Y,-K)$ represents the inverse of $(Y,K)$ in the homology knot concordance group. For an enriched $\mathcal{S}$--complex $\mathfrak{E}$, we denote its dual enriched $\mathcal{S}$--complex by $\mathfrak{E}^{\dagger}$. In particular, $(\mathfrak{E}^{\omega}(Y,K))^{\dagger}$ is locally equivalent to $\mathfrak{E}^{\omega}(-Y,-K)$.

\begin{lem}[{\cite[Proposition~7.13]{DISST}}]\label{lem:Omega-detected-by-R}
An enriched $\mathcal{S}$--complex $\mathfrak{E}^{\omega}(Y,K)$ is weakly locally equivalent to the trivial enriched $\mathcal{S}$--complex if and only if
$$
r_0^{\omega}\bigl(\mathfrak{E}^{\omega}(Y,K)\bigr)=\infty
\qquad\text{and}\qquad
r_0^{\omega}\bigl((\mathfrak{E}^{\omega}(Y,K))^{\dagger}\bigr)=\infty.
$$
\end{lem}

We next record some basic classes for which this nontriviality
condition holds.
\begin{prp}\label{prp:basic-nontrivial-Omega}
Let $K$ be a knot in $S^3$. Suppose that $K$ is one of the following:
\begin{itemize}
\item A knot representing an element of infinite order in the algebraic knot concordance group;
\item A squeezed knot with nonzero Rasmussen invariant $s(K)$.
\end{itemize}
Then there exists $\omega\in(0,\frac12)$ such that
$$
\Omega^{\mathfrak{E},\omega}([K])\neq[0].
$$
In the second case, one may take $\omega=\frac14$. Moreover, for such a choice of $\omega$, at least one of
$$
r_0^{\omega}(K),
\qquad
r_0^{\omega}(-K)
$$
is finite.
\end{prp}

\begin{proof}
We first suppose that $K$ has infinite order in the algebraic knot concordance group. Then there exists $\omega\in(0,\frac12)$ such that $e^{4\pi i\omega}$ is not a root of $\Delta_K(t)$ and
$$
\sigma_K(e^{4\pi i\omega})\neq0.
$$
For this value of $\omega$, \cite[Corollary~7.8]{DISST} shows that the Fr{\o}yshov invariant of the $\omega$--holonomy complex is nonzero; see also \cite[Proposition~7.9]{DISST}. Hence the enriched $\mathcal{S}$--complex $\mathfrak{E}^{\omega}_K$ is not weakly locally equivalent to the trivial enriched complex. Therefore,
$$
\Omega^{\mathfrak{E},\omega}([K])\neq[0].
$$

We next consider the case of squeezed knots. Recall that the concordance homomorphism
\[
\widetilde{s} \colon \mathcal{C} \to \Z
\]
given in \cite{DISST} is defined using the traceless meridional holonomy condition, corresponding to $\omega=\frac14$. Since slice-torus invariants agree for squeezed knots~\cite{FellerLewarkLobbSqueezed}, and since $2\widetilde{s}$ is a slice-torus invariant by \cite[Theorem~4.24]{DISST}, we have
$$
\widetilde{s}(K)
=
\frac{s(K)}{2}
\neq0.
$$
Since the invariant $\widetilde{s}$ is defined in terms of $\Omega^{\mathfrak{E},\frac{1}{4}}$ forgetting the filtration structure, it follows from \cite[Theorem~5.56]{DISST} that at least one of
$$
r_0^{1/4}(K),
\qquad
r_0^{1/4}(-K)
$$
is finite. Therefore, Lemma~\ref{lem:Omega-detected-by-R} implies that
$$
\Omega^{\mathfrak{E},1/4}([K])\neq[0].
$$
Finally, the assertion concerning $r_0^{\omega}$ follows directly from Lemma~\ref{lem:Omega-detected-by-R}, completing the proof.
\end{proof}

The behavior of the Chern--Simons filtration under negative definite cobordisms gives the following monotonicity property, which is the general-holonomy analogue of \cite[Corollary~5.54]{DISST}.

\begin{thm}[{\cite[Theorem~7.10]{DISST}}]\label{thm:cob_ineq}
Let $K\subset Y$ and $K'\subset Y'$ be knots in homology $3$--spheres, and suppose that
$$
\Delta_{Y,K}(e^{4\pi i\omega})\neq0,
\qquad
\Delta_{Y',K'}(e^{4\pi i\omega})\neq0.
$$
Let $W$ be a smooth compact connected cobordism from $Y$ to $Y'$, and
let $S\subset W$ be a smoothly and properly embedded oriented surface
from $K$ to $K'$. Suppose that $(W,S)$ is negative definite of strong
height $0$ with respect to the holonomy parameter $\omega$. Then, for every $s\in[-\infty,0]$,
$$
r_s^\omega(Y,K)
\leq
r_{s-2\kappa_{\min}^\omega(W,S)}^\omega(Y',K')
+
2\kappa_{\min}^\omega(W,S).
$$
Moreover, suppose that
$$
r_s^\omega(Y,K)<\infty,
\qquad
\kappa_{\min}^\omega(W,S)=0,
$$
and that equality holds. Then there exists an irreducible $SU(2)$--representation
$$
\rho\colon
\pi_1(W\smallsetminus S)
\longrightarrow
SU(2)
$$
such that the image of a meridian of $S$ lies in the conjugacy class of
$$
\begin{pmatrix}
e^{2\pi i\omega} & 0\\
0 & e^{-2\pi i\omega}
\end{pmatrix}.
$$
\end{thm}
\begin{rmk}
There is a typo in \cite[Corollary~5.54]{DISST}, where the inequality is stated as
$$
r^{\frac{1}{4}}_{s-2\kappa_{\min}^{\frac{1}{4}}(W,S)}(Y,K)
\leq
r_{s}^{\frac{1}{4}}(Y',K')
+
2\kappa_{\min}^{\frac{1}{4}}(W,S).
$$
The correct inequality is
$$
r^{\frac{1}{4}}_{s}(Y,K)
\leq
r_{s-2\kappa_{\min}^{\frac{1}{4}}(W,S)}^{\frac{1}{4}}(Y',K')
+
2\kappa_{\min}^{\frac{1}{4}}(W,S),
$$
as follows from \cite[Theorem~5.50]{DISST}.
\end{rmk}

\section{Applications to unknotting and Dehn surgery}

\subsection{Obstructions to same-sign null-homologous twists}

Throughout this subsection, we specialize to knots in $S^3$ and use the abbreviations
$$
\mathfrak{E}^{\omega}_K
:=
\mathfrak{E}^{\omega}(S^3,K),
\qquad
r_s^\omega(K)
:=
r_s^\omega(S^3,K).
$$We remark that our proof extends to knots in a homology $3$--sphere $Y$ for which $Y\mathbin{\#}-Y$ bounds a simply connected definite $4$--manifold. For example, this applies whenever $Y$ bounds a contractible $4$--manifold, as is the case for many Brieskorn spheres~\cite{Akbulut-Kirby:1979, Casson-Harer:1981, Fickle:1984}. For simplicity, however, we state the result only for $Y=S^3$. The following proposition is the knot-theoretic obstruction which will be used
in the proof of Theorem~\ref{thm:main1}.

\begin{prp}\label{prp:annulus-cyclic-obstruction}
Let $K_1$ and $K_2$ be knots in $S^3$. Let $W$ be a smooth compact connected
negative-definite cobordism from $S^3$ to $S^3$, and let $A\subset W$ be a
smoothly and properly embedded null-homologous annulus from $K_1$ to $K_2$.
Suppose that
$
\chi(W)+\sigma(W)=0.
$
Suppose that $K_1$ and $K_2$ are concordant and that
$
r_0^\omega(K_1)<\infty
$
for some $\omega\in(0,\frac12)$. Then
$
\pi_1(W\smallsetminus A)
$
cannot be infinite cyclic.
\end{prp}

\begin{proof}
Suppose, for a contradiction, that
$
\pi_1(W\smallsetminus A)\cong\mathbb{Z}.
$ We first claim that $(W,A)$ is negative definite of strong height $0$ and that
$$
\kappa_{\min}^\omega(W,A)=0.
$$
Since $A$ is null-homologous,
$$
A\cdot A=0,
\qquad
c_1(L)\cdot A=0
$$
for every complex line bundle $L$ over $W$. Moreover,
$$
\chi(W)+\sigma(W)=0,
$$
and $K_1$ and $K_2$ are concordant, so their Levine--Tristram signatures agree
at $e^{4\pi i\omega}$. Consequently, the index formula reduces to
$$
\operatorname{ind}(A_L)
=
8\kappa(A_L)-1,
\qquad
\kappa(A_L)=-c_1(L)^2.
$$
Since the intersection form of $W$ is negative definite, $\kappa(A_L)\geq0$,
with equality if and only if $c_1(L)=0$. Thus the unique reducible of minimal
index has index $-1$ and zero energy. Therefore,
$$
\kappa_{\min}^\omega(W,A)=0.
$$
The corresponding element is
$$
\eta_\omega(W,A)=1,
$$
which is invertible in $R_\omega$. Hence $(W,A)$ is negative definite of strong
height $0$.

Since $K_1$ is concordant to $K_2$, we have
$$
r_0^\omega(K_1)=r_0^\omega(K_2).
$$
On the other hand, applying Theorem~\ref{thm:cob_ineq} to $(W,A)$ and using
$\kappa_{\min}^\omega(W,A)=0$ gives
$$
r_0^\omega(K_1)\leq r_0^\omega(K_2).
$$
Thus equality holds in the cobordism inequality. Since
$r_0^\omega(K_1)<\infty$, the equality case of Theorem~\ref{thm:cob_ineq}
produces an irreducible representation
$$
\rho\colon
\pi_1(W\smallsetminus A)
\longrightarrow
SU(2).
$$
This is impossible because $\pi_1(W\smallsetminus A)\cong\mathbb{Z}$ and every
$SU(2)$--representation of a cyclic group is reducible, which completes the
proof.
\end{proof}

We are now ready to prove Theorem~\ref{thm:main1}, the first of our main results. In fact, we prove the following stronger statement, formulated in terms of null-homologous twists.

\begin{thm}\label{mainthmbody}
Let $K$ be a slice knot, and suppose that there is a ribbon concordance from
$K_1\# K_2$ to $K$. If one of the summands, say $K_1$, satisfies
$
\Omega^{\mathfrak{E}, \omega}([K_1]) \neq [0]
$
for some $\omega$ satisfying \eqref{admissible}, then any sequence of null-homologous twists
transforming $K$ into the unknot must contain both positive and negative null-homologous twists.

In particular, if a slice knot $K$ is obtained as a band sum of two nontrivial
knots and one of the summands represents a nontrivial class in
$\Omega^{\mathfrak{E}, \omega}$, then $K$ has untwisting number greater than
one.
\end{thm}

\begin{proof}
We first rule out a sequence of positive null-homologous twists. Suppose, for a
contradiction, that $K$ can be transformed into the unknot by such a sequence of
$n$ twists. By Proposition~\ref{prp:positive-twists-disk} and the given ribbon
concordance, $K_1\mathbin{\#}K_2$ bounds a smooth properly embedded
null-homologous disk
$$
D\subset
\bigl(\#_n\overline{\mathbb{CP}}^{\,2}\bigr)^\circ
$$
whose complement has infinite cyclic fundamental group.

Since $K_1\mathbin{\#}K_2$ is concordant to the slice knot $K$, it is also
slice. Therefore, $-K_1$ is concordant to $K_2$. By
Lemma~\ref{lem:Omega-detected-by-R}, at least one of
$$
r_0^\omega(K_1),
\qquad
r_0^\omega(-K_1)
$$
is finite. If $r_0^\omega(-K_1)<\infty$, then
$r_0^\omega(K_2)<\infty$. Moreover,
$$
\Omega^{\mathfrak{E},\omega}([K_2])
=
\Omega^{\mathfrak{E},\omega}([-K_1])
=
-\Omega^{\mathfrak{E},\omega}([K_1])
\neq[0].
$$
Thus, after interchanging $K_1$ and $K_2$ if necessary, we may assume that
$$
r_0^\omega(K_1)<\infty.
$$

Using the standard pair-of-pants cobordism in a collar of the boundary, we obtain a
null-homologous embedded annulus
$$
A\subset W,
\qquad
W\cong
\bigl(\#_n\overline{\mathbb{CP}}^{\,2}\bigr)
\smallsetminus
\bigl(\mathring{B}^4\sqcup\mathring{B}^4\bigr),
$$
from $K_1$ to $-K_2$. A straightforward application of the
Seifert--van Kampen theorem shows that
$$
\pi_1(W\smallsetminus A)\cong\mathbb{Z}.
$$
Moreover, $W$ is negative definite and
$
\chi(W)+\sigma(W)=0.
$
Since $K_1$ is concordant to $-K_2$, all the hypotheses of
Proposition~\ref{prp:annulus-cyclic-obstruction} are satisfied.
This contradicts Proposition~\ref{prp:annulus-cyclic-obstruction}. Hence $K$
cannot be transformed into the unknot by positive null-homologous twists alone.

It remains to rule out negative null-homologous twists. If such a sequence
existed, then $-K$ could be transformed into the unknot by positive
null-homologous twists. The same argument as above then applies, and we
conclude that every sequence of null-homologous twists transforming $K$ into the
unknot must contain twists of both signs, which completes the proof.
\end{proof}

We now prove Corollary~\ref{cor:whitehead-double}, whose statement we recall:

\begin{repcorollary}{cor:whitehead-double}
Let $K$ be a non-slice strongly quasipositive knot, and set
$J_K=\Wh(K)\#-\Wh(K)$. Then any sequence of null-homologous twists
transforming $J_K$ into the unknot must contain both positive and negative
twists. In particular, $tu(J_K)=2$.
\end{repcorollary}

\begin{proof}
By \cite{Rudolph1993}, the positive Whitehead double $\Wh(K)$ is strongly quasipositive, so Proposition~\ref{prp:basic-nontrivial-Omega} gives
$$
\Omega^{\mathfrak{E},1/4}([\Wh(K)])\neq[0].
$$
On the other hand, $J_K$ is slice. Applying Theorem~\ref{mainthmbody} with $K_1=\Wh(K)$ and $K_2=-\Wh(K)$ shows that every sequence of null-homologous twists transforming $J_K$ into the unknot must contain both positive and negative twists. Thus $tu(J_K)\geq2$.

Finally, the standard diagram of $\Wh(K)$ contains a clasp crossing whose change unknots $\Wh(K)$. The corresponding clasp crossing of $-\Wh(K)$ has the opposite sign. Changing these two crossings transforms $J_K$ into the unknot. Therefore, $tu(J_K)=2$.
\end{proof}

\subsection{Obstructions to Dehn surgery number one}

We now specialize the preceding discussion to the case
$
\omega=\frac14
$
and to pairs of the form $(Y,U)$, where $Y$ is a homology $3$--sphere and $U\subset Y$ is an unknot contained in a ball. Since
$
\Delta_{Y,U}(t)=1,
$
the non-root condition is automatic. We write
$
\mathfrak{E}_Y
:=
\mathfrak{E}^{1/4}(Y,U)
$
for the associated enriched $\mathcal{S}$--complex after the specialization $T=1$, and abbreviate
$$
r_s(Y)
:=
r_s^{1/4}(Y,U).
$$
The homology concordance construction of \cite{DISST} gives a homomorphism
$$
\Omega\colon
\Theta^3_{\mathbb{Z}}
\longrightarrow
\Theta^{\mathfrak{E}}
; \qquad
[Y] \longmapsto  [\mathfrak{E}_Y].
$$
By the duality property of the enriched $\mathcal{S}$--complex, $\mathfrak{E}_{-Y}$ represents the dual of $\mathfrak{E}_Y$. Hence, Lemma~\ref{lem:Omega-detected-by-R} implies that if
$\Omega([Y])\neq[0]$, then
\[
r_0(Y)<\infty
\qquad\text{ or }\qquad
r_0(-Y)<\infty.
\]

The following consequence of Theorem~\ref{thm:cob_ineq} is the only
cobordism property needed below.

\begin{lem}\label{lem:r0-definite-cobordism}
Let $W$ be a smooth compact connected negative-definite cobordism from a homology $3$--sphere $Y_0$ to a homology $3$--sphere $Y_1$. Suppose that
$
[Y_0]=[Y_1]\in\Theta^3_{\mathbb{Z}}
$
and
$
r_0(Y_0)<\infty.
$
Then $W$ cannot be simply connected.
\end{lem}

\begin{proof}
Suppose, for a contradiction, that $W$ is simply connected. Choose a properly embedded path from $Y_0$ to $Y_1$. Transporting a standard unknot along this path gives a trivial annulus $A\subset W$ from an unknot $U_0\subset Y_0$ to an unknot $U_1\subset Y_1$. Thus $(W,A)$ is a cobordism of pairs from $(Y_0,U_0)$ to $(Y_1,U_1)$. A straightforward application of the Seifert--van Kampen theorem shows that
$$
\pi_1(W\smallsetminus A)\cong\mathbb{Z}.
$$

We claim that $(W,A)$ is negative definite of strong height $0$ and that
$\kappa_{\min}^{1/4}(W,A)=0$.
Since $A$ is a trivial annulus,
$$
\chi(A)=0,
\qquad
A\cdot A=0.
$$
Moreover, the boundary signature terms vanish and
$$
\chi(W)+\sigma(W)=0.
$$
Since $\omega=\frac14$, the monopole-number term in the reducible-index formula also vanishes. Consequently, the index formula reduces to
$$
\operatorname{ind}(A_L)
=
8\kappa(A_L)-1,
\qquad
\kappa(A_L)=-c_1(L)^2.
$$
Since the intersection form of $W$ is negative definite, $\kappa(A_L)\geq0$, with equality if and only if $c_1(L)=0$. Thus the unique reducible of minimal index has index $-1$ and zero energy. Therefore,
$$
\kappa_{\min}^{1/4}(W,A)=0.
$$
The associated element $\eta_{1/4}(W,A)$ is a unit after the specialization $T=1$. Hence $(W,A)$ is negative definite of strong height $0$.

Since
$
[Y_0]=[Y_1]
$
in $\Theta^3_{\mathbb{Z}}$, their enriched complexes are weakly locally equivalent, and hence
$$
r_0(Y_0)=r_0(Y_1).
$$
On the other hand, applying Theorem~\ref{thm:cob_ineq} to $(W,A)$ and using $\kappa_{\min}^{1/4}(W,A)=0$ gives
$$
r_0(Y_0)\leq r_0(Y_1).
$$
Thus equality holds in the cobordism inequality. Since $r_0(Y_0)<\infty$, the equality case of Theorem~\ref{thm:cob_ineq} produces an irreducible representation
$$
\rho\colon
\pi_1(W\smallsetminus A)
\longrightarrow
SU(2).
$$
This is impossible because $\pi_1(W\smallsetminus A)\cong\mathbb{Z}$ and every $SU(2)$--representation of a cyclic group is reducible, which completes the proof.
\end{proof}

We are now ready to prove the surgery obstruction stated in the introduction, whose statement we recall for convenience.

\begin{reptheorem}{thm:DS-ribbon-obstruction}
Let $Y$ be a homology $3$--sphere representing the trivial element in
$\Theta^3_\Z$, and  suppose that there is a ribbon homology cobordism from
$Y_1 \mathbin{\#} Y_2$ to $Y$.  If one of the summands, say $Y_1$, satisfies
$\Omega([Y_1])\neq [0]$,
then $Y$ cannot be obtained by surgery on a knot in
$S^3$.
\end{reptheorem}

\begin{proof}
Suppose, for a contradiction, that $Y$ is obtained by surgery on a knot in $S^3$. Then there exists a smooth compact simply connected definite $4$--manifold $X$ with
$\partial X=-Y$;
see~\cite{Auckly,AucklyHyperbolic}.
Let $R$ be a ribbon homology cobordism from $Y_1\mathbin{\#}Y_2$ to $Y$. Since $[Y]=0$ in $\Theta^3_{\mathbb{Z}}$, we have $[Y_2]=[-Y_1]$. Glue $X$ to $R$ along $Y$ and set $V:=R\cup_YX$. Then $\partial V=-(Y_1\mathbin{\#}Y_2)$. Since $R$ is a homology cobordism, the intersection form of $V$ is identified with that of $X$. In particular, $V$ is definite. Moreover, by Remark~\ref{rmk:ribbon-fillings}, the manifold $V$ is simply connected.

Next, attach a $3$--handle to $V$ along the separating $2$--sphere in
$\partial V=-(Y_1\mathbin{\#}Y_2)$
which realizes the connected-sum decomposition. Denote the resulting $4$--manifold by $\widetilde{V}$. Then
$\partial\widetilde{V}=-Y_1\sqcup-Y_2$.
Attaching a $3$--handle changes neither the fundamental group nor the intersection form. Thus $\widetilde{V}$ is simply connected and definite.

By the assumption $\Omega([Y_1])\neq[0]$, we have
$$r_0(Y_1)<\infty
\qquad\text{ or }\qquad
r_0(-Y_1)<\infty.
$$
If $\widetilde{V}$ is negative definite, its possible incoming ends are $Y_1$ and $Y_2$. Since
$[Y_2]=[-Y_1]$,
one of these choices has finite $r_0$. The same argument applies if $\widetilde{V}$ is positive definite, after reversing its orientation.

Consequently, in all cases we obtain a smooth compact simply connected negative-definite cobordism satisfying the hypotheses of Lemma~\ref{lem:r0-definite-cobordism}, a contradiction. This completes the proof.
\end{proof}

We now prove Theorem~\ref{thm:example}, whose statement we recall.

\begin{reptheorem}{thm:example}
Let $Y$ be a homology $3$--sphere. Suppose that $Y$ is one of the following:
\begin{itemize}
\item A Seifert homology $3$--sphere with positive Fintushel--Stern
$R$--invariant \cite{FintushelStern1985}, or more generally, a nontrivial integral linear combination of
Seifert homology $3$--spheres detected by~\cite[Corollary~2.1]{Furuta1990HomologyCobordism};
\item A nontrivial integral linear combination of homology $3$--spheres of the form
$\{S^3_{1/n}(K)\}_{n\in  \Z_{>0}}$, where $K$ is a squeezed knot with positive Rasmussen invariant.
\end{itemize}
Then $\Omega([Y])\neq [0]$.
\end{reptheorem}

\begin{proof}
We first consider the Seifert case. Suppose that $[Y]$ is a nontrivial
integral linear combination of Seifert homology $3$--spheres detected by
\cite[Corollary~2.1]{Furuta1990HomologyCobordism}. More precisely, write
\[
[Y]
=
\sum_{k=1}^{N} c_k[\Sigma_k]
\in \Theta^3_{\mathbb Z},
\]
where $c_1,\ldots,c_N\in\mathbb Z$ are not all zero and
\[
\Sigma_k=\Sigma(a_{k,1},\ldots,a_{k,n_k}),
\qquad 1\leq k\leq N,
\]
are Seifert homology $3$--spheres satisfying
\[
R(\Sigma_k)>0
\]
for every $k$, and
\[
a_{1,1}\cdots a_{1,n_1}
<
a_{2,1}\cdots a_{2,n_2}
<
\cdots
<
a_{N,1}\cdots a_{N,n_N}.
\]

For every oriented homology $3$--sphere $M$, the invariant
\[
r_0^{\mathrm{NST}}(M)\in[0,\infty]
\]
was defined in \cite{NST24} as a $3$--manifold analogue of
\eqref{equation_r_s} and was used to recover Furuta's theorem
\cite[Corollary~2.1]{Furuta1990HomologyCobordism}. By
\cite[Proof of Corollary~5.5]{NST24}, at least one of
$r_0^{\mathrm{NST}}(Y)$ and $r_0^{\mathrm{NST}}(-Y)$ is finite.

On the other hand, the comparison result
\cite[Equation~(120)]{DISST} gives
\[
r_0(M)\leq 2r_0^{\mathrm{NST}}(M)
\]
for every homology $3$--sphere $M$. The factor of $2$ reflects the
different normalizations of the Chern--Simons functional used in
\cite{DISST} and \cite{NST24}. Consequently, at least one of $r_0(Y)$
and $r_0(-Y)$ is finite.\footnote{More precisely, let
$S\subset [0,1]\times Y$
be an unknotted properly embedded disk capping off the unknot $U_1\subset Y$. Then $([0,1]\times Y,S)$ is a cobordism of pairs from $(Y,U_1)$ to $(Y,\varnothing)$. Comparing the associated filtration-preserving chain map and its dual with the reducible-counting maps yields the comparison of the $r_0$-invariants; see \cite[Equation~(119)]{DISST} and \cite[Proposition~5.3]{DS20}.} Consequently, at least one of $r_0(Y)$ and $r_0(-Y)$ is finite. By
Lemma~\ref{lem:Omega-detected-by-R}, we conclude that
\[
\Omega([Y])\neq[0].
\]

We next consider the surgery family. Let $K$ be a squeezed knot with
$s(K)>0$. Since all slice-torus invariants agree on squeezed
knots~\cite{FellerLewarkLobbSqueezed} and $2\widetilde{s}$ is a
slice-torus invariant~\cite[Theorem~4.24]{DISST}, we have
\[
\widetilde{s}(K)=\frac{1}{2}s(K)>0.
\]
It follows that
\[
h(S^3_{+1}(K),U)<0,
\]
where $U$ is an unknot in $S^3_{+1}(K)$ and $h$ denotes the
Fr{\o}yshov-type invariant defined in~\cite{DS19}. For $n\geq 1$, set
\[
Y_n:=S^3_{1/n}(K).
\]

By \cite[Equation~(20)]{DISST}, the Fr{\o}yshov invariant factors
through the local equivalence class of the underlying
$\mathcal{S}$--complex. Since $h(Y_1,U)<0$, it follows that
$\mathfrak{E}_{Y_1}$ is not weakly locally equivalent to the trivial
enriched $\mathcal{S}$--complex. Therefore,
Lemma~\ref{lem:Omega-detected-by-R} implies that
\[
r_0(Y_1)<\infty
\qquad\text{ or }\qquad
r_0(-Y_1)<\infty.
\]
On the other hand, $Y_1=S^3_{+1}(K)$ bounds the simply connected
positive-definite $4$--manifold given by the $1$--framed trace of
$K$. After reversing orientation and equipping the resulting
cobordism with the standard null-homologous annulus between the
boundary unknots, we obtain a cobordism of pairs of strong height
zero. Thus, \cite[Corollary~5.54]{DISST} gives
$r_0(-Y_1)=\infty$.
It follows that
$r_0(Y_1)<\infty$.
Moreover, the cobordism argument in the proof of
\cite[Theorem~5.12]{NST24}, applied analogously to the corresponding
pairs with unknots, implies that
\begin{itemize}
    \item $\infty>r_0(Y_1)>r_0(Y_2)>\cdots$;
    \item $r_0(-Y_i)=\infty$ for all $i$.
\end{itemize}
Let
\[
[Y]=\sum_{n=1}^{N}a_n[Y_n]\in\Theta^3_{\mathbb{Z}}
\]
be a nonzero finite linear combination with $a_N\neq 0$. If $a_N>0$,
then \cite[Proposition~5.55]{DISST}, applied to the corresponding
pairs with unknots, implies that
$r_0(Y)<\infty$.
Thus, Lemma~\ref{lem:Omega-detected-by-R} implies that $\Omega([Y])\neq[0]$.
If $a_N<0$, we apply the same argument to $-Y$. Since the coefficient
of $Y_N$ in $[-Y]$ is $-a_N>0$, we obtain
$r_0(-Y)<\infty$.
Again, Lemma~\ref{lem:Omega-detected-by-R} implies that
$\Omega([Y])\neq[0]$.
This proves the second family and completes the proof.
\end{proof}

\bibliographystyle{alpha}
\bibliography{prive}

{\small
\bigskip

\noindent\textsc{Department of Mathematical Sciences, KAIST, Daejeon, Republic of Korea}\\
\textit{Email address}: \texttt{himori@kaist.ac.kr}

\medskip

\noindent\textsc{Department of Mathematical Sciences, KAIST, Daejeon, Republic of Korea}\\
\textit{Email address}: \texttt{jungpark0817@kaist.ac.kr}

\medskip

\noindent\textsc{Department of Mathematics, Graduate School of Science, Kyoto University, Kyoto, Japan}\\
\textit{Email address}: \texttt{masaki.taniguchi@math.kyoto-u.ac.jp}
\par}

\end{document}